%% file: bullins_lai_arxiv.tex
\documentclass{article} 
\usepackage{hyperref,url}
\usepackage{natbib}
\usepackage{mathrsfs} 
\usepackage{amsmath,amsbsy,amsfonts,amssymb,amsthm,dsfont,units}
\usepackage{algorithm,algorithmic}
\usepackage{mathtools}
\usepackage{color,cases}
\usepackage{tikz}
\usetikzlibrary{calc,shapes}
\usepackage{caption}
\usepackage{subcaption}
\usepackage[utf8]{inputenc}
\usepackage[english]{babel}
\usepackage{multirow}
\usepackage{bbm}
\usepackage{xr}
\usepackage{fullpage}
\usepackage{graphicx}
\usepackage{enumitem}
\usepackage{titling}
\usepackage{cleveref} 

\newtheorem{theorem}{Theorem}[section]

\newtheorem{definition}[theorem]{Definition}
\newtheorem{lemma}[theorem]{Lemma}

\newtheorem{remark}[theorem]{Remark}
\include{defs}

\newcommand{\HOM}{\textsc{HigherOrderMirrorProx}}

\newcommand{\BSg}{$\textsc{BinarySearch}_\gamma$}

\newcommand{\TT}{\mathcal{T}}

\title{Higher-order methods for convex-concave min-max optimization and monotone variational inequalities}

\author{Brian Bullins\\Toyota Technological Institute at Chicago\\ \texttt{bbullins@ttic.edu} \and Kevin A. Lai\\
	Georgia Institute of Technology\\
	\texttt{kevinlai@gatech.edu}}

\begin{document}

    \maketitle

\begin{abstract}%
We provide improved convergence rates for constrained convex-concave min-max problems and monotone variational inequalities with higher-order smoothness. In min-max settings where the $p^{th}$-order derivatives are Lipschitz continuous, we give an algorithm \HOM{} that achieves an iteration complexity of $O(1/T^{\frac{p+1}{2}})$ when given access to an oracle for finding a fixed point of a $p^{th}$-order equation. We give analogous rates for the weak monotone variational inequality problem. For $p>2$, our results improve upon the iteration complexity of the first-order Mirror Prox method of \cite{nemirovski2004prox} and the second-order method of \cite{monteiro2012iteration}. We further instantiate our entire algorithm in the unconstrained $p=2$ case.
\end{abstract}

\section{Introduction}
In this work, we focus on two well-studied classes of problems: monotone variational inequalities (MVIs) and convex-concave min-max problems  \citep{minty1962monotone,kinderlehrer1980introduction,nemirovski2004prox}. In an MVI, we are given a \emph{monotone} operator $F:\Z\ra \R^n$ over a convex set $\Z\subseteq \R^n$, and the goal is to find a point $z^*\in \Z$ such that
\begin{align}\label{eq:VI}
	\forall z \in \Z, \inner{F(z),z^*-z} \le 0.
\end{align}
Such a point is called a solution to a weak (Minty) MVI \citep{komlosi1999stampacchia}. The MVI problem \cref{eq:VI} is closely related to the classic min-max optimization problem:
\begin{align}\label{eq:minmax}
	\min_{x\in\X} \max_{y\in \Y} g(x,y)
\end{align}
where $g:\X\times \Y\ra \R$ is a convex-concave function over convex sets $\X$ and $\Y$. Such problems are ubiquitous in statistics, optimization, machine learning, and game theory. Solving \cref{eq:minmax} is equivalent to finding the Nash Equilibrium of a zero-sum game and is also sometimes called a \emph{saddle point problem}.

The Mirror Prox (MP) algorithm of \cite{nemirovski2004prox} is a popular method for solving both \cref{eq:VI} (when $F$ is Lipschitz continuous) and \cref{eq:minmax} (when $g$ is smooth). MP is a generalization of the extragradient algorithm of \cite{korpelevich1976extra}, and it converges in $O(1/T)$ iterations, which is tight for \emph{first-order methods} (FOMs) \citep{nemirovski1983problem}. Given that MP achieves the optimal performance for FOMs, there is a natural question of whether one can improve the iteration complexity by using \textit{higher-order methods} (HOMs), which tend to converge in fewer iterations but at the expense of higher cost per iteration.  HOMs use higher-order derivatives of the objective function and generally require \emph{higher-order smoothness}, namely that the higher-order derivatives of the objective be Lipschitz continuous. 

In convex and nonconvex optimization, while FOMs such as gradient descent are the gold standard for optimization algorithms, HOMs are useful in a variety of different settings. Newton's method is one of the most well-known HOMs, and it is a central component of path-following interior-point methods \citep{nesterov1994interior}. In cases when the higher-order update is efficiently computable, HOMs can achieve faster overall running times than FOMs. For example, HOMs have been used to find approximate local minima in nonconvex optimization faster than gradient descent \citep{agarwal2017finding,carmon2018accelerated}. While second-order methods are the most common type of HOM, there has also been significant recent work on HOMs beyond second-order methods \citep{agarwal2018lower,arjevani2018oracle, gasnikovdec2018global, jiang2018optimal, bubeck2018near, bullins2018fast}.

HOMs have seen much less study in the context of MVIs and min-max problems. \cite{monteiro2012iteration} use a second-order method with an \textit{implicit} update that achieves an improved iteration complexity of $O(1/T^{\frac{3}{2}})$ for problems with second-order smoothness. Their method uses the Hybrid Proximal Extragradient (HPE) framework established in \cite{monteiro2010complexity} and requires access to an oracle for finding a fixed point of a constrained second-order equation. However, it was unknown whether one could achieve further improved iteration complexity in the presence of third-order smoothness and beyond.

\paragraph{Contributions.}
Our main contribution is a higher-order method \HOM{} for approximately solving MVIs and convex-concave min-max problems that achieves an iteration complexity of $O(1/T^{\frac{p+1}{2}})$ for problems with $p^{th}$-order smoothness. To our knowledge, this is the first work showing that improved convergence rates are possible for problems with third-order smoothness and beyond. Our algorithm requires access to an oracle for finding a fixed point of a $p^{th}$-order equation, using a higher-order implicit update that can be thought of as a generalization of Mirror Prox. Since the implicit update may be difficult to compute in the constrained case, we show how to instantiate our algorithm in the second-order unconstrained case, giving overall running time bounds in that setting.\\

\noindent We begin by reviewing definitions, notions of convergence, and related work in \Cref{sec:prelims}. Then we summarize our main results and our algorithm in \Cref{sec:results}. In \Cref{sec:main}, we present the proof of our main result. We then show how to fully instantiate our algorithm in the unconstrained $p=2$ case in \Cref{sec:explicit}.

\section{Preliminaries}\label{sec:prelims}
We will use MVI($F,\Z$) to denote the MVI given in \cref{eq:VI} over a vector field $F:\Z\ra \R^n$ and convex constraint set $\Z\subseteq \R^n$. Unless otherwise specified, we will use $z^*$ to signify a solution to MVI($F,\Z$). Throughout the paper, we will use $\gamma_t$ to represent positive weights, and we let $\Gamma_T \defeq \sum_{t=1}^T \gamma_t$. We use $\nabla$ to denote the Jacobian operator. We use $\norm{\cdot}$ to denote an arbitrary norm and $\norm{\cdot}_*$ to denote its dual norm. We use $\norm{\cdot}_2$ to denote the Euclidean norm for vectors and the operator norm for matrices.

We use $D:\Z \times \Z\ra \R$ to denote a Bregman divergence over a distance generating function $d:\Z\ra \R$ that is 1-strongly convex with respect to some norm $\norm{\cdot}$. Recall that the definition of a Bregman divergence is as follows:
\begin{align}
D(u,v) = d(u) - d(v)  - \inner{\nabla d(v),u-v}
\end{align}
for all $u,v\in \Z$.

We now discuss several key definitions:
\begin{definition}
A vector field $F:\Z \ra \R^n$ is \emph{monotone} if $\inner{F(u)-F(v),u-v} \ge 0$ for all $u,v \in \Z$.
\end{definition}
For notational convenience, we assume our algorithms have access to a monotone operator $F$. This is the usual assumption in MVIs, but it will also allow us to solve min-max problems, as we now show. For min-max problems \cref{eq:minmax}, one can consider the \emph{gradient descent-ascent field} of $g$:
\begin{equation}\label{eq:saddlevf}
F_g(x,y) \defeq \left(\begin{array}{c}
\grad_x g(x,y)\\
-\grad_y g(x,y)
\end{array}\right)
\end{equation}
Letting $z = \begin{pmatrix}
x\\y
\end{pmatrix}$ and $\Z = \X\times \Y$, we can say $F_g$ maps $\Z$ to $\R^n$ with only a slight abuse of notation. It is then easy to show that $F_g$ is monotone when $g$ is convex-concave. So to apply our algorithms to min-max settings, we simply apply them on $F_g$.

Our algorithms will we require the following general notion of smoothness:
\begin{definition}\label{def:smooth}
A vector field $F:\Z \ra \R^n$ is $p^{th}$-order $L_p$ smooth w.r.t. $\norm{\cdot}$ if, for all $u,v \in \Z$,
\begin{equation*}
\norm{\nabla^{p-1} F(u) - \nabla^{p-1} F(v)}_* \le L_p \norm{u-v},
\end{equation*}
where we define
$\norm{\nabla^{p-1} F(u) - \nabla^{p-1} F(v)}_* \defeq \max\limits_{h : \norm{h} \leq 1} \Bigl|\nabla^{p-1} F(u) [h]^{p-1} - \nabla^{p-1} F(v) [h]^{p-1}\Bigr|.$
\end{definition}
\begin{remark}\label{rem:smooth}
\noindent Our definition of $p^{th}$-order smoothness as a property of the $(p-1)^{th}$ derivative of $F$ is motivated by the min-max setting \cref{eq:minmax}, where $F_g$ is already expressed in terms of the gradient of $g$. If $F_g$ is $p^{th}$ order smooth, this is a statement about the Lipschitz continuity of $p^{th}$ order derivatives of $g$.	
\end{remark}

Another key component of our algorithms is the \textit{$p^{th}$-order Taylor expansion} of $F$ at $u$ evaluated at $v$:
\begin{align}
\TT_p(v;u) = \sum_{i=0}^p \frac{1}{i!}\nabla^{(i)} F(u)[v-u]^i.
\end{align}
While $\TT$ depends on $F$, we leave this implicit to lighten notation, as the relevant $F$ will be clear from context.

\begin{remark}\label{rem:order}
To be consistent with Remark~\ref{rem:smooth}, when we refer to ``$p^{th}$-order methods,'' we will be referring to methods that use a $(p-1)^{th}$-order Taylor expansion of $F$ and which typically require $p^{th}$-order smoothness. Again, this indexing makes sense in the context of min-max problems, where a $p^{th}$-order method uses a Taylor expansion involving $p^{th}$-order derivatives of $g$.
\end{remark}

A well-studied consequence of Definition~\ref{def:smooth} is the following:
\begin{fact}
	Let $u,v\in \Z$, and let $F:\Z\ra \R^n$ be $p^{th}$-order $L_p$ smooth. Then,
	\begin{equation}\label{eq:smoothtaylordiff}
	\norm{F(v) - \TT_{p-1}(v;u)}_* \le \frac{L_{p}}{p!} \norm{v-u}^{p}.
	\end{equation}
\end{fact}

Finally, our algorithms will all require the following assumption:
\begin{assumption}\label{as:mvi-soln}
	There exists a solution $x^*\in \X$ to the weak variational inequality MVI$(F,\X)$, namely $x^*$ is a point that satisfies \cref{eq:VI}.
\end{assumption}
Assumption~\ref{as:mvi-soln} always holds when $\Z$ is a compact convex set and $F$ is continuous on $\Z$ \citep{kinderlehrer1980introduction}.

\subsection{Notions of convergence for variational inequalities}
The main solution concept for \cref{eq:VI} that we consider is an \textit{$\eps$-approximate weak solution} to MVI($F,\Z$), namely a point $z^*$ such that:
\begin{align}\label{eq:apx-weak}
\forall z \in \Z, \inner{F(z),z^*-z} \le \eps.
\end{align}

Our main bounds will be of the form: 
\begin{align}\label{eq:sum_soln}
\forall z\in \Z, \frac{1}{\Gamma_T}\sum_{t=1}^T \gamma_t \inner{F(z_t),z_t-z} \le \eps,
\end{align}
where $z_t$ are iterates produced by our algorithm, $\gamma_t$ are positive constants, and $\Gamma_T = \sum_t \gamma_t$. We now show conditions under which a guarantee of the form \cref{eq:sum_soln} gives $\eps$-approximate weak solutions.

\begin{lemma}\label{lem:minmaxavg}
	Let $F:\Z\ra \R^n$, let $z_t \in \Z$ for $t\in [T]$ be monotone, and let $\gamma_t > 0$. Let $\bar{z}_t = \frac{1}{\Gamma_T}\sum_{t=1}^T \gamma_t z_t$. Assume \cref{eq:sum_soln} holds. Then $\bar{z}_t$ is an $\eps$-approximate weak solution to MVI($F,\Z$).
\end{lemma}
\begin{proof} By monotonicity, we have:
	\begin{align*}
	\inner{F(z_t),z_t-z} \ge \inner{F(z),z_t-z}.
	\end{align*}
	Therefore,
	\begin{align*}
	\sum_{t=1}^T \gamma_t \inner{F(z_t),z_t-z} \ge \sum_{t=1}^T \gamma_t \inner{F(z),z_t-z} = \Gamma_T \inner{F(z), \bar{z}_t -z}.
	\end{align*}
	Then $\bar{z}$ is an $\eps$-approximate solution to the weak MVI problem.
\end{proof}

\subsection{Solving convex-concave min-max problems with variational inequalities}
The classic notion of convergence for \cref{eq:minmax} is the \emph{duality gap} $\Phi_{\X \times \Y}:\X\times \Y \ra \R$:
\begin{align}\label{eq:duality_gap}
	\Phi_{\X \times \Y}(x,y) = \max_{\hat{y}\in \Y} g(x,\hat{y}) - \min_{\hat{x} \in \X} g(\hat{x},y).
\end{align}
The duality gap is defined in terms of a min-max objective $g$, but we leave it implicit because the relevant $g$ will be clear from context. We will now show how to prove bounds on the duality gap given a bound like in \cref{eq:sum_soln}.

We will use the following lemma to prove bounds on the duality gap:
\begin{lemma}\label{lem:minmax}
	Let $F:\Z\ra \R^n$, let $z_t \in \Z$ for $t\in [T]$, and let $\gamma_t > 0$. Let $\bar{z}_t = \frac{1}{\Gamma_T}\sum_{t=1}^T \gamma_t z_t$. Assume \cref{eq:sum_soln} holds. If $F$ is the gradient descent-ascent field for a convex-concave problem (as in \cref{eq:saddlevf}), then $\Phi_{\X \times \Y}(\bar{z}_t)\le \eps$.
\end{lemma}
\begin{proof}
	When $F$ is the gradient descent-ascent field for a convex-concave problem, we have:
	\begin{align*}
		\inner{F(z_t),z_t - z} &= (\inner{\nabla_x g(x_t,y_t),x_t - x} + \inner{-\nabla_y g(x_t, y_t), y_t - y})\\
		&\ge g(x_t,y_t) - g(x,y_t) + g(x_t,y) - g(x_t,y_t)\\
		&= g(x_t,y) - g(x,y_t).
	\end{align*}
	Overall, we then have:
	\begin{align*}
		\sum_{t=1}^T \gamma_t \inner{F(z_t),z_t - z} \ge \sum_{t=1}^T \gamma_t (g(x_t,y) - g(x,y_t)) \ge \Gamma_T(g(\bar{x}_t,y) - g(x,\bar{y}_t)) \ge  \Gamma_T \Phi_{\X \times \Y}(\bar{x}_t,\bar{y}_t).
	\end{align*}
\end{proof}

\subsection{Related work}
\paragraph{Monotone variational inequalities.}
The weak MVI \cref{eq:VI} is a classic and well-studied optimization problem \citep{minty1962monotone,komlosi1999stampacchia,nemirovski2004prox,monteiro2010complexity}. It is closely related to the strong MVI problem \citep{stampacchia1970variational}, where the goal is to find a $z^* \in \Z$ such that
\begin{align}\label{eq:strong_VI}
	\forall z \in \Z, \inner{F(z^*),z^*-z} \le 0.
\end{align}
When $F$ is continuous and single-valued, any solution to the weak MVI \cref{eq:VI} is a solution to the strong MVI.

Our algorithm is based on the Mirror Prox (MP) algorithm of \cite{nemirovski2004prox}, which is a generalization of the extragradient method of \cite{korpelevich1976extra}. MP is a first-order method that achieves $O(1/T)$ iteration complexity, which is tight \citep{nemirovski1983problem}. \cite{monteiro2010complexity} prove convergence rates for MP in the unconstrained case by formulating MP as an instance of what they call a Hybrid Proximal Extragradient (HPE) algorithm. \cite{monteiro2012iteration} provide a second-order algorithm to solve \cref{eq:VI} in settings with second-order smoothness. That algorithm achieves an $O(1/T^{\frac{3}{2}})$ iteration complexity, and its analysis goes through the HPE framework from \cite{monteiro2010complexity}.

\paragraph{Min-max optimization.} Many convex-concave min-max optimization problems are either solved with MP or first-order no-regret algorithms. \cite{ouyang2018lower} show a lower bound of $\Omega(1/T)$ for first-order methods in constrained smooth convex-concave saddle point problems, even in the simple case when $g(x,y)=f(x) + \inner{Ax-b,y} - h(y)$ for convex $f$ and $h$. A number of recent works have also applied second-order methods to unconstrained smooth min-max problems, where the second-order information is often accessed through Hessian-vector products \citep{balduzzi2018mechanics,gemp2018global,letcher2018stable,adolphs2018local,abernethy2019last,schafer2019competitive}.

\paragraph{Higher-order methods for convex optimization.} Higher-order methods have a long history of use in solving convex optimization problems. Assuming Lipschitz continuity of the Hessian, \cite{nesterov2008accelerating} provided an accelerated variant of the cubic regularization method~\citep{nesterov2006cubic}, which was further generalized by \cite{baes2009estimate} under $p^{th}$-order smoothness assumptions. The rate in \citep{nesterov2008accelerating} was later improved by \cite{monteiro2013accelerated}, and since then several works concerning lower bounds in this setting \citep{agarwal2018lower, arjevani2018oracle} have shown that this rate is essentially tight (up to logarithmic factors) when the Hessian is Lipschitz continuous. Recently, several works have shown that the lower bound is also essentially tight for $p > 2$ \citep{gasnikovdec2018global, jiang2018optimal, bubeck2018near, bullins2018fast}, leading to advances in related problems, such as $\ell_\infty$ regression \citep{bullins2019higher} and parallel non-smooth convex optimization \citep{bubeck2019complexity}.

\section{Main results}\label{sec:results}
Our main result is a new higher-order method \HOM{} (\Cref{alg:main}) for solving MVIs and convex-concave min-max problems with higher-order smoothness. We prove the following convergence rate:
\begin{theorem}
Suppose $F:\Z\ra \R^n$ is $p^{th}$-order $L_p$-smooth. Let $R \defeq \max\limits_{z\in \Z} D(z, z_1)$. Moreover,\\ let $\eps = \frac{16L_{p}}{p!}\pa{\frac{R}{T}}^{\frac{p+1}{2}}$. Then for $\bar{z}_T$ as output by \Cref{alg:main}:
\begin{enumerate}
	\item If $F$ is monotone, then $\bar{z}_T$ is an $\eps$-approximate solution to the weak MVI problem.
	\item If $F$ is the gradient descent-ascent field for a convex-concave problem over $\X$ and $\Y$, then $\Phi_{\X \times \Y}(\bar{z}_t)\le \eps$.
\end{enumerate}
\end{theorem}

Our result matches the rate of \cite{monteiro2012iteration} when $p=2$ and gives improved convergence rates for higher $p$. To our knowledge, this is the first algorithm to achieve improved iteration complexity in the presence of higher-order smoothness. We compare our algorithm to that of \cite{monteiro2012iteration} in more detail in \Cref{sec:compare_alg}.

Similar to other higher-order algorithms, which require an oracle for solving a minimization over a $p^{th}$ order Taylor series \citep{gasnikovdec2018global, jiang2018optimal, bubeck2018near}, our algorithm requires an oracle for solving a fixed point problem of a $p^{th}$ order equation. While this oracle is stronger, we believe it is justified given that the MVI and convex-concave min-max settings are significantly more difficult compared to convex minimization problems. A common downside of higher-order algorithms is that the required oracle may be difficult to compute, particularly in the constrained setting. We can also consider running our algorithm in the unconstrained setting, which requires a slightly weaker unconstrained oracle rather than a constrained oracle. We discuss how to interpret our bounds in the unconstrained setting in \Cref{sec:unconstrained}.

Finally, we show how to instantiate our method in the second-order unconstrained case, giving the following running time bounds:
\begin{theorem}[Main theorem, $p=2$ (Informal)] Suppose $F:\R^n\ra \R^n$ is sufficiently smooth, and let $\braces{(\hat{z}_t, \gamma_t)}_{t\in[T]}$ be the output of \HOM{} $(p=2)\ +$ \BSg{} (Algorithm \ref{alg:fomp}). Then, for $\Gamma_T \defeq \sum\limits_{t=1}^T \gamma_t$, the iterates $\braces{\hat{z}_t}_{t\in[T]}$ satisfy, for all $z \in \reals^n$,
\begin{equation}
\frac{1}{\Gamma_T}\sum\limits_{t=1}^T \inner{\gamma_t F(\hat{z}_t), \hat{z}_t - z} \leq 8L_{2}\pa{\frac{\max\braces{D(z,z_1), 1}}{T}}^{\frac{3}{2}},
\end{equation}
with per-iteration cost dominated by $\tilde{O}(1)$ matrix inversions.\footnote{Here we use the $\tilde{O}(\cdot)$ notation to suppress logarithmic factors.}
\end{theorem}

\begin{algorithm}[h]
	\caption{\HOM}
	
	\begin{algorithmic}\label{alg:main}
		\STATE {\bfseries Input:} $z_1 \in \Z$, $p\ge 1$, $0 < \eps < 1$, $T>0$
		\FOR{$t=1$ {\bfseries to} $T$}
		\STATE Determine $\gamma_t$, $\hat{z}_t$ such that:
		\STATE \begin{equation}\label{eq:implicitzhat1} \hat{z}_t = \argmin\limits_{z \in \Z} \braces{\gamma_t\inner{\TT_p(\hat{z}_t;z_t), z - z_t} + D(z, z_t)}, \text{ and }\end{equation}
		\begin{equation}\label{eq:implicitzhat2}\frac{p!}{32L_p\norm{\hat{z}_t - z_t}^{p-1}} \leq \gamma_t \leq \frac{p!}{16L_p\norm{\hat{z}_t - z_t}^{p-1}}\end{equation}
		\STATE \begin{equation}\label{eq:bregt1}z_{t+1} = \argmin\limits_{z \in \Z} \braces{\inner{\gamma_t F(\hat{z}_t), z - \hat{z}_t} + D(z, z_t)}\end{equation}
		\ENDFOR
		\STATE Define $\Gamma_T \defeq \sum_{t=1}^T \gamma_t$
		\RETURN $\bar{z}_T \defeq \frac{1}{\Gamma_T} \sum_{t=1}^T \gamma_t \hat{z}_t$
		
	\end{algorithmic}
\end{algorithm}

\subsection{Interpreting our results in the unconstrained setting}\label{sec:unconstrained}
In the unconstrained setting, the standard solution concepts for MVIs and min-max problems can be vacuous in general. For example, for $g(x,y) = x^\top y$ and the associated vector field $F_g$, all approximate solutions to the min-max problem / MVI are exact solutions. However, the bounds we prove are still meaningful. In the MVI case, our guarantee can be interpreted as stating that for all $z$ such that $D(z,z_1)\le R$, we have $\inner{F(z), \bar{z}_T - z} \le O(R/T^{\frac{p+1}{2}})$ as long as $D(z^*,z_1) \le R$. Likewise, for min-max problems, if $\Z'$ is a convex set containing $z^*$, then we can say that $\Phi_{\Z'}(\bar{z}_T) \le O(R/T^{\frac{p+1}{2}})$, where $R \ge \max_{z \in \Z'} D(z,z_1)$.

\subsection{Explanation of our algorithm}
Our algorithm is inspired by the Mirror Prox (MP) algorithm of \cite{nemirovski2004prox}, defined as follows:
 \begin{equation}\label{eq:mp1}
 \hat{z}_t = \argmin\limits_{z\in \Z} \braces{\inner{\gamma_t F(z_t), z - z_t} + D(z, z_t)}
 \end{equation}
 \begin{equation}\label{eq:mp2}
 z_{t+1} = \argmin\limits_{z\in \Z} \braces{\inner{\gamma_t F(\hat{z}_t), z - \hat{z}_t} + D(z, z_t)}
 \end{equation}
 where $D$ is a Bregman divergence. \cite{nemirovski2004prox} motivates MP with a ``conceptual prox method'', which is given as follows:
\begin{align}\label{eq:prox}
	z_{t+1} = \arg \min_{z\in \Z} \{ \inner{\gamma_{t+1} F(z_{t+1}), z-  z_{t+1} } + D(z,z_t)\}.
\end{align}
This is an \textit{implicit} method, as computing $z_{t+1}$ requires solving the equation above for a given step-size $\gamma_{t+1}$. However, this method has good iteration complexity. \cite{nemirovski2004prox} shows that if one could run \cref{eq:prox} exactly, then the $\gamma$-averaged iterate $z_T = \frac{1}{\Gamma_T}\sum_{t=1}^T \gamma_t z_t$ converges at a rate of $O(1/\Gamma_T)$. Thus, if one could implement \cref{eq:prox} with large step-sizes, one could achieve faster iteration complexity.

It turns out that as long as one approximates \cref{eq:prox} with small error, one can achieve a similar convergence rate. The MP algorithm with constant $\gamma_t$ does just that, leading to a $O(1/T)$ convergence rate. While one would like to increase the step-size in MP to improve the convergence rate, this approach does not work because MP with large step-sizes will no longer approximate \cref{eq:prox} with small error. 

In our algorithm, we replace the first-order minimization in MP \cref{eq:mp1} with a $p^{th}$-order minimization \eqref{eq:implicitzhat1}. We also simultaneously choose a particular step-size. This can be viewed as approximating \cref{eq:prox} with large step-sizes while using the higher-order minimization to ensure that our algorithm is still a ``good'' approximation of \cref{eq:prox}.

\subsection{Comparison to \cite{monteiro2012iteration}}
\label{sec:compare_alg}
\cite{monteiro2012iteration} give a second-order algorithm for solving \cref{eq:VI} with iteration complexity $O(1/T^{\frac{3}{2}})$ in the presence of second-order smoothness. Like our algorithm, their algorithm also heavily relies on the idea of approximating a proximal point method with a large step-size. In fact, their algorithm is very similar to our algorithm in the second-order case. However, our analysis is rather different and arguably simpler. While their analysis goes through the Hybrid Proximal Extragradient framework of \cite{monteiro2010complexity}, our analysis relies on a natural extension of the Mirror Prox analysis. Finally, \cite{monteiro2012iteration} only deal with the Euclidean setting, whereas we allow arbitrary norms.

While \cite{monteiro2012iteration} do not explicitly instantiate their second-order oracle, they mention that their oracle reduces to solving a strongly monotone variational inequality, which can then be solving using a variety of approaches, including interior point methods. In the $p=2$ case, our oracle can be similarly instantiated.

\section{Higher-Order Mirror Prox Guarantees}\label{sec:main}
In this section, we prove our main result of the convergence guarantees provided by Algorithm \ref{alg:main}.
\begin{lemma}\label{thm:mainho} Suppose $F:\R^n \ra \R^n$ is $p^{th}$-order $L_p$-smooth and let $\Gamma_T \defeq \sum\limits_{t=1}^T \gamma_t$. Then, the iterates $\braces{\hat{z}_t}_{t\in[T]}$ generated by Algorithm \ref{alg:main} satisfy, for all $z \in \Z$,
\begin{equation}
\frac{1}{\Gamma_T}\sum\limits_{t=1}^T \inner{\gamma_t F(\hat{z}_t), \hat{z}_t - z} \leq \frac{16L_{p}}{p!}\pa{\frac{D(z,z_1)}{T}}^{\frac{p+1}{2}}.
\end{equation}
\end{lemma}
\begin{theorem}\label{thm:main}
	Suppose $F:\Z\ra \R^n$ is $p^{th}$-order $L_p$-smooth. Let $R \defeq \max\limits_{z\in \Z} D(z, z_1)$. Moreover,\\ let $\eps = \frac{16L_{p}}{p!}\pa{\frac{R}{T}}^{\frac{p+1}{2}}$. Then for $\bar{z}_T$ as output by \Cref{alg:main}:
	\begin{enumerate}
		\item If $F$ is monotone, then $\bar{z}_T$ is an $\eps$-approximate solution to the weak MVI problem.
	\item If $F$ is the gradient descent-ascent field for a convex-concave problem over $\X$ and $\Y$, then $\Phi_{\X \times \Y}(\bar{z}_t)\le \eps$.
	\end{enumerate}
\end{theorem}

\Cref{thm:main} follows immediately from Lemmas~\ref{lem:minmaxavg}, \ref{lem:minmax}, and \ref{thm:mainho}. To prove Lemma~\ref{thm:mainho}, we will need to establish our main technical result (Lemma~\ref{lem:higher-order_mp}), which we prove in \Cref{sec:higher-order_mp} and whose proof proceeds in a similar manner to the Mirror Prox analysis~\citep{nemirovski2004prox, tseng2008accelerated}.

\begin{lemma}\label{lem:higher-order_mp}
	Suppose $F:\R^n \ra \R^n$ is $p^{th}$-order $L_p$-smooth. Then, $\braces{\gamma_t, \hat{z}_t, z_{t+1}}_{t\in[T]}$ as generated by Algorithm \ref{alg:main} satisfy, for all $z \in \Z$,
	\begin{equation}\label{eq:sum-bound}
	\sum\limits_{t=1}^T \inner{\gamma_t F(\hat{z}_t), \hat{z}_t - z} + \frac{1}{4}\sum\limits_{t=1}^T\norm{\hat{z}_t-z_t}^2 + \frac{1}{4}\sum_{t=1}^T \norm{z_{t+1} - \hat{z}_t}^2  \leq D(z, z_1) - D(z,z_{t+1}).
	\end{equation}
\end{lemma}
We will also need the following technical lemma:
\begin{lemma}\label{lem:apply_power_mean}
	Let $R, \xi_t \geq 0$ for all $t \in \bra{T}$, and let $\sum_{t=1}^T \xi^2 \le R$. Then $\sum_{t=1}^T \xi^{-p} \ge \frac{T^{\frac{p}{2}+1}}{R^{\frac{p}{2}}}$.
\end{lemma}
\begin{proof}\label{app:apply_power_mean}
	We use the following power means:
	\begin{align*}
		M_{1}(x) &= \frac{\sum_{t=1}^T x_t}{T}\\
		M_{-2/p}(x) &= \left(\frac{\sum_{t=1}^T x_t^{-2/p}}{T}\right)^{-p/2}.
	\end{align*}
	By the power mean inequality, we have $M_{1/p}(x) \ge M_{-2/p}(x)$, so letting $x_t = \frac{1}{\xi_t^p}$ gives:
	\begin{align*}
		\frac{\sum_{t=1}^T \frac{1}{\xi_t^p}}{T} &\ge \left(\frac{T}{\sum_{t=1}^T \xi_t^2}\right)^{p/2} \ge \pa{\frac{T}{R}}^{p/2}\\
		&\Rightarrow \sum_{t=1}^T \frac{1}{\xi_t^p} \ge \frac{T^{1+p/2}}{R^{p/2}}.
	\end{align*}
\end{proof}
We now have the necessary tools to prove Lemma~\ref{thm:mainho}.
\begin{proof}[Proof of Lemma~\ref{thm:mainho}]
	Using Lemma~\ref{lem:higher-order_mp}, we can divide both sides of \eqref{eq:sum-bound} by $\Gamma_T$, and so using the non-negativity of $\|\cdot\|$ and the Bregman divergence, we get:
	\begin{align*}
		\frac{1}{\Gamma_T}\sum\limits_{t=1}^T \inner{\gamma_t F(\hat{z}_t), \hat{z}_t - z}  \leq \frac{D(z,z_1)}{\Gamma_T}.
	\end{align*}
We simply need to lower bound $\frac{1}{\Gamma_T}$ in order to prove our convergence rate result. By Assumption~\ref{as:mvi-soln}, we know that there exists a solution $z^*$ to MVI($F,\Z$), which means that for all $t\in [T]$, we have $\inner{\gamma_t F(\hat{z}_t), \hat{z}_t - z^*} \ge 0$. We can combine this with Lemma~\ref{lem:higher-order_mp} to get that $\frac{1}{4}\sum_{t=1}^T \norm{\hat{z}_t - z_t} \le D(z^*,z_1)$. Since $\gamma_t \geq \frac{p!}{32L_{p} \norm{\hat{z}_t - z_t}^{p-1}}$, we can apply Lemma~\ref{lem:apply_power_mean} by setting $\xi_t = \norm{\hat{z}_t - z_t}$ and $R = D(z^*,z_1)$, which gives the result.	
\end{proof}

\subsection{Proof of main technical result (Lemma~\ref{lem:higher-order_mp})}\label{sec:higher-order_mp}

Before proving Lemma~\ref{lem:higher-order_mp}, we state a useful lemma concerning the updates \eqref{eq:implicitzhat1} and \eqref{eq:bregt1} in Algorithm \ref{alg:main}.
\begin{lemma}[\cite{tseng2008accelerated}]\label{lem:bregmanineq}
	Let $\phi(\cdot)$ be a convex function, let $z \in \Z$, and let
	\begin{equation}
	z_+ = \argmin\limits_{x} \braces{\phi(x) + D(x, z)}.
	\end{equation}
	Then, for all $x \in \Z$,
	\begin{equation}
	\phi(x) + D(x, z) \geq \phi(z_+) + D(z_+, z) + D(x, z_+).
	\end{equation}
\end{lemma}
\begin{proof}
	By the optimality condition for $z_+$, we know that for all $x \in \Z$, 
	\begin{equation}
	\phi(x) + \inner{\grad_x D(z_+, z), x - z_+} \geq \phi(z_+).
	\end{equation}
	Rearranging and adding $D(x, z)$ to both sides gives us
	\begin{align*}
		\phi(x) + D(x, z) &\geq \phi(z_+) + D(x,z) - \inner{\grad_x D(z_+, z), x - z_+}\\
		&= \phi(z_+) + D(x,z) + D(x, z_+) + D(z_+, z) - D(x,z)\\
		&= \phi(z_+) + D(x, z_+) + D(z_+, z),
	\end{align*}
	where the first equality comes from the Bregman three-point property, i.e.,
	\begin{equation}
	\inner{\grad d(w) - \grad d(v), u-v} = D(u,v) + D(v,w) - D(u,w),\ \text{ for all }\ u,v,w \in \Z.
	\end{equation}
\end{proof}
We now prove Lemma~\ref{lem:higher-order_mp}, which is our main technical result.

\begin{proof}[Proof of Lemma \ref{lem:higher-order_mp}]
	By Lemma~\ref{lem:bregmanineq}, along with the algorithm's determination of $\hat{z}_t$, we have that for all $z \in \Z$,
	\begin{equation}\label{eq:tsenghatx}
	\gamma_t\inner{\TT_{p-1}(\hat{z}_t;z_t), \hat{z}_t - z} \le D(z,z_t) - D(z,\hat{z}_t) - D(\hat{z}_t,z_t)
	\end{equation}
	Using Lemma~\ref{lem:bregmanineq} again with the choice of $z_{t+1}$, it follows that for all $z \in \Z$,
	\begin{equation}\label{eq:tsengxtp1}
	\gamma_t \inner{F(\hat{z}_t), z_{t+1} - z} \le D(z,z_t) - D(z,z_{t+1}) - D(z_{t+1},z_t).
	\end{equation}
	We may now observe that
	\begin{align*}
	\gamma_t \inner{&F(\hat{z}_t), \hat{z}_t - z} = \gamma_t \inner{F(\hat{z}_t), \hat{z}_t - z_{t+1}} + \gamma_t\inner{ F(\hat{z}_t), z_{t+1} - z}\\
	&= \gamma_t\inner{ F(\hat{z}_t) - \TT_{p-1}(\hat{z}_t;z_t), \hat{z}_t - z_{t+1}} + \gamma_t\inner{ \TT_{p-1}(\hat{z}_t;z_t), \hat{z}_t - z_{t+1}} + \gamma_t\inner{ F(\hat{z}_t), z_{t+1} - z}\\
	&\leq \gamma_t \inner{F(\hat{z}_t) - \TT_{p-1}(\hat{z}_t;z_t), \hat{z}_t - z_{t+1}} - D(z_{t+1}, \hat{z}_t) - D(\hat{z}_t, z_t)  + D(z, z_t) - D(z, z_{t+1}),
	\end{align*}
	where the final inequality follows from \eqref{eq:tsenghatx} and \eqref{eq:tsengxtp1}.
	Now by H\"older's inequality, using eq.~\eqref{eq:smoothtaylordiff}, and the 1-strong convexity of $d(\cdot)$ w.r.t. $\norm{\cdot}$, it follows that
	\begin{align*}
	\gamma_t\inner{& F(\hat{z}_t), \hat{z}_t - z} \leq \gamma_t \norm{F(\hat{z}_t) - \TT_{p-1}(\hat{z}_t;z_t)}_{*}\cdot\norm{\hat{z}_t - z_{t+1}} - D(z_{t+1}, \hat{z}_t) - D(\hat{z}_t, z_t) \\ &\hspace{7em}+ D(z, z_t) - D(z, z_{t+1})\\
	&\leq \frac{\gamma_t L_{p}}{p!}\norm{\hat{z}_t - z_t}^{p}\cdot\norm{\hat{z}_t - z_{t+1}} - D(z_{t+1}, \hat{z}_t) - D(\hat{z}_t, z_t)  + D(z, z_t) - D(z, z_{t+1})\\
	&\leq \frac{\gamma_tL_{p}}{p!}\norm{\hat{z}_t - z_t}^{p} \cdot \norm{\hat{z}_t - z_{t+1}} - \frac{1}{2}\norm{z_{t+1} -  \hat{z}_t}^2 - \frac{1}{2}\norm{\hat{z}_t - z_t}^2  + D(z, z_t) - D(z, z_{t+1}).
	\end{align*}
	
	Finally, by our guarantee from Algorithm \ref{alg:main} that $\gamma_t \leq \frac{p!}{16 L_{p}\norm{\hat{z}_t - z_t}^{p-1}}$, and using the fact that $ab \leq \frac{a^2}{2} + \frac{b^2}{2}$ for $a, b \geq 0$, it follows that 
	\begin{equation}\label{eq:term_to_sum}
	\gamma_t\inner{ F(\hat{z}_t), \hat{z}_t - z} + \frac{1}{4}\norm{\hat{z}_t-z_t}^2 + \frac{1}{4}\norm{z_{t+1} - \hat{z}_t}^2\leq D(z, z_t) - D(z, z_{t+1}).
	\end{equation}
	Summing over $t = 1,\dots,T$ gives the result.
\end{proof}

\section{Instantiating \HOM{} (for $p=2$)}\label{sec:explicit}
In this section, we provide an efficient implementation of \HOM for the case where $F$ is second-order smooth. In particular, we consider the unconstrained problem (i.e., $\Z = \reals^n$) with the Bregman divergence chosen as $D(u,v) = \frac{1}{2}\norm{u-v}_2^2$. First, for technical reasons, we require the following assumption:

\begin{assumption}\label{assum:invert} During the execution of Algorithm~\ref{alg:fomp}, for all $t\geq 1$, $\gamma > 0$, we assume that $(\bI + \gamma \nabla F(z_t))$ is invertible and $\sigma_{\min}(\gamma^{-1}\bI + \nabla F(z_t)) \geq \sigma_{\min}(\nabla F(z_t))$.
\end{assumption}

As we later discuss in \Cref{sec:inversion}, these conditions always hold for convex-concave min-max problems. We then arrive at the following result for this setting:

\begin{theorem}[Main theorem, $p=2$]\label{thm:mainp2} Suppose $F:\R^n\ra \R^n$ is first-order $L_1$-smooth, second-order $L_2$-smooth, and Assumption~\ref{assum:invert} holds. Let $z^*$ be a solution to MVI($F, \R^n$), let $T > 0$, and let $\braces{(\hat{z}_t, \gamma_t)}_{t\in[T]}$ be the output of \HOM{} $(p=2)$ + \BSg{} (Algorithm \ref{alg:fomp}). Further assume that, for all $t$, $\sigma_{\min}(\nabla F(z_t)) \geq \mu$. Then, for $\Gamma_T \defeq \sum\limits_{t=1}^T \gamma_t$, the iterates $\braces{\hat{z}_t}_{t\in[T]}$ satisfy, for all $z \in \reals^n$,
\begin{equation}
\frac{1}{\Gamma_T}\sum\limits_{t=1}^T \inner{\gamma_t F(\hat{z}_t), \hat{z}_t - z} \leq 8L_{2}\pa{\frac{\max\braces{D(z,z_1), 1}}{T}}^{\frac{3}{2}}.
\end{equation}
In addition, the computational cost of each iteration of Algorithm \ref{alg:fomp} is dominated by a total of\\$O\pa{\log \pa{\frac{L_1\norm{z_1 - z^*}_2T}{\mu}}}$ matrix inversions.
\end{theorem}
\begin{proof}
We will first show that the choices of $\gamma_+$ and $\gamma_-$ are valid binary search bounds whenever\\ \BSg{} is called by Algorithm \ref{alg:fomp}, i.e., that $\gamma_+ \geq \frac{1}{12\norm{\hat{z}_{t}(\gamma_+) - z_t}_2}$ and $\gamma_- \leq \frac{1}{12\norm{\hat{z}_{t}(\gamma_-) - z_t}_2}$. We begin with our choice of $\gamma_+ = T^{\frac{3}{2}}$. Suppose that, for some iteration $t$, it is the case that $\gamma_+ < \frac{1}{8\norm{\hat{z}_{t}(\gamma_+) - z_t}_2}$. If so, then the algorithm sets $\gamma_t \gets \gamma_+$, which means that $\Gamma_T \geq \gamma_+ = T^{\frac{3}{2}}$. Therefore, since we know that
\begin{equation}
\frac{1}{\Gamma_T}\sum\limits_{t=1}^T \inner{\gamma_t F(\hat{z}_t), \hat{z}_t - z} \leq 8L_{2}\frac{D(z,z_1)}{\Gamma_T},
\end{equation}
it follows that
\begin{equation}
\frac{1}{\Gamma_T}\sum\limits_{t=1}^T \inner{\gamma_t F(\hat{z}_t), \hat{z}_t - z} \leq 8L_{2}\frac{D(z,z_1)}{T^{\frac{3}{2}}} \leq 8L_{2}\frac{D(z,z_1)}{T^{\frac{3}{2}}} \leq 8L_{2}\pa{\frac{\max\braces{D(z,z_1), 1}}{T}}^{\frac{3}{2}},
\end{equation}
and so we would be done. In addition, supposing it is the case that $\gamma_- \geq \gamma_+$ (at which point, the algorithm sets $\gamma_t \gets \gamma_-$), we again reach this conclusion by the same reasoning. For ensuring the validity of $\gamma_-$, note that by \eqref{eq:q-lower-bound}, it follows that $\gamma_- = \delta \leq \frac{1}{12\norm{\hat{z}_{t}(\delta) - z_t}_2}$.

Having established the validity of the binary search bounds in the case that the search routine is in fact called, we now move on to show how we may explicitly instantiate the implicitly defined update in \eqref{eq:implicitzhat1}. Namely, in this setting the key conditions \eqref{eq:implicitzhat1} and \eqref{eq:implicitzhat2} that must simultaneously hold can be equivalently expressed as~
\begin{equation}\label{eq:lipschitzp1} \hat{z}_t = \argmin\limits_{z \in \reals^d} \braces{\gamma_t\inner{F(z_t) + \nabla F(z_t)(\hat{z}_t-z_t), z - z_t} + \frac{1}{2}\norm{z - z_t}^2}, \text{ and } \end{equation}
\begin{equation}\label{eq:lipschitzp2}
\frac{1}{16L_1\norm{\hat{z}_t - z_t}_2} \leq \gamma_t \leq \frac{1}{8L_1\norm{\hat{z}_t - z_t}_2}.\end{equation}

From \eqref{eq:lipschitzp1}, it follows by first-order optimality conditions that $\gamma_t(F(z_t) + \nabla F(z_t)(\hat{z}_t-z_t)) + \hat{z}_t - z_t = 0$, and so rearranging gives us
\begin{align*}
(\bI + \gamma_t \nabla F(z_t))\hat{z}_t = (\bI + \gamma_t \nabla F(z_t))z_t - \gamma_t F(z_t).
\end{align*}
Since we assume that $(\bI + \gamma_t \nabla F(z_t))$ is invertible, it follows that
\begin{equation}
\hat{z}_t = z_t - \gamma_t(\bI + \gamma_t \nabla F(z_t))^{-1} F(z_t),
\end{equation}
which is precisely the update that occurs in Algorithm \ref{alg:fomp}. All that remains is to ensure that we may determine $\gamma_t$ such that \eqref{eq:lipschitzp2} holds, which follows from the output of \BSg{} as a consequence of Lemma~\ref{lem:bsearch}. Finally, since the iteration complexity of \BSg{} is bounded by 
\begin{equation}
N = O\pa{\log\pa{\frac{\bar{C}T}{\delta}}} = O\pa{\log\pa{\frac{T\norm{F(z_t)}_2}{\sigma_{\min}(\nabla F(z_t))}}} \leq O\pa{\log \pa{\frac{L_1\norm{z_1 - z^*}_2T}{\mu}}},
\end{equation}
where the final inequality follows from Lemma~\ref{lem:F_bound}, which bounds $\norm{F(z_t)}$, along with our assumption that, for all $t$, $\sigma_{\min}(\nabla F(z_t)) \geq \mu$, and each iteration of \BSg{} requires $O\pa{\log \pa{\frac{L_1\norm{z_1 - z^*}_2T}{\mu}}}$ matrix inversions, which results in the total complexity in the theorem.
\end{proof}

\begin{algorithm}[h]
	\caption{\HOM{} $(p=2)$ + \BSg}
	
	\begin{algorithmic}\label{alg:fomp}
		\STATE {\bfseries Input:} $z_1 \in \reals^n$, $0 < \eps < 1$, $T>0$
		\FOR{$t=1$ {\bfseries to} $T$}
			\STATE Set $\gamma_- = \frac{\sigma_{\min}(\nabla F(z_t))}{12\norm{F(z_t)}_2}$, $\gamma_+ = T^{\frac{3}{2}}$
			\IF{$\gamma_+ < \frac{1}{8\norm{\hat{z}_t(\gamma_+) - z_t}_2}$}
			\STATE $\gamma_t \leftarrow \gamma_+$
			\ELSIF {$\gamma_- \geq \gamma_+$}
			\STATE $\gamma_t \leftarrow \gamma_-$
			\ELSE 
			\STATE $\gamma_t \leftarrow \text{\BSg}(z_t, \eps, \gamma_-, \gamma_+)$
			\ENDIF
			\STATE $\hat{z}_{t} \defeq z_t - \gamma_t(\bI + \gamma_t\nabla F(z_t))^{-1}F(z_t)$
			\STATE $z_{t+1} = \argmin\limits_z \braces{\inner{\gamma_t F(\hat{z}_t), z - \hat{z}_t} + D(z, z_t)}$
		\ENDFOR
		\STATE Define $\Gamma_T \defeq \sum\limits_{t=1}^T \gamma_t$
		\RETURN $\bar{z}_T \defeq \frac{1}{\Gamma_T}\sum\limits_{t=1}^T \gamma_t \hat{z}_t$
	\end{algorithmic}
\end{algorithm}

\subsection{Binary search}
The following lemmas establish the correctness of the main binary search procedure.
\begin{lemma}\label{lem:bsearch}
 Suppose $\gamma_-, \gamma_+$ are such that $\gamma_- \leq \frac{1}{12\norm{\hat{z}_{t}(\gamma_-) - z_t}_2}$ and $\gamma_+ \geq \frac{1}{12\norm{\hat{z}_{t}(\gamma_+) - z_t}_2}$, where $\hat{z}_{t}(\gamma) = z_t - \gamma(\bI + \gamma\nabla F(z_t))^{-1}F(z_t)$, Then, \BSg{} (Algorithm \ref{alg:gbs}) outputs $\bar{\gamma}$ such that
 \begin{equation}
\frac{1}{16\norm{\hat{z}_t(\bar{\gamma}) - z_t}_2} \leq \bar{\gamma} \leq \frac{1}{8\norm{\hat{z}_t(\bar{\gamma}) - z_t}_2}
 \end{equation}
 after $N = O\pa{\log\pa{\frac{\bar{C}T}{\delta}}}$ iterations of the binary search procedure, where $\delta$, $\bar{C}$ are as defined in the algorithm.
\end{lemma}
\begin{proof}
By assumption, we have that $\gamma_-$ and $\gamma_+$ are initialized to be valid search bounds, i.e., $\gamma_- \leq \frac{1}{12\norme{\hat{z}_{t}(\gamma_-) - z_t}}$ and $\gamma_+ \geq \frac{1}{12\norme{\hat{z}_{t}(\gamma_+) - z_t}}$.
By Lemma \ref{lem:lipschitz-q} and letting $\bar{C} \defeq \max\braces{C,1}$, we know that, for all $x, y \geq \gamma_-$,
\begin{equation}\label{eq:lipschitz-q}
\abs{q(y) - q(x)} \leq \bar{C}\cdot\abs{y-x}
\end{equation} 
After $N = O\pa{\log\pa{\frac{\bar{C}T}{\delta}}}$ iterations of the binary search procedure we know that $\abs{\gamma_+ - \gamma_-} \leq \frac{\delta}{100\bar{C}} \leq \frac{\delta}{100}$, and so taken together with \eqref{eq:lipschitz-q}, we have
\begin{align*}
\gamma_+ &\leq \gamma_- + \frac{\delta}{100} \leq q(\gamma_-) + \frac{\delta}{100} \leq q(\gamma_+) + C\abs{\gamma_+ - \gamma_-} + \frac{\delta}{100} \leq q(\gamma_+) + \frac{2\delta}{100} \\
&\leq \frac{3}{2}q(\gamma_+) = \frac{1}{8\norme{\hat{z}_t(\gamma_+) - z_t}}.
\end{align*}
Here, the last inequality follows from the fact that, for $\gamma > 0$,
\begin{align}
q(\gamma) = \frac{1}{12\norme{(\gamma^{-1}\bI + \nabla F(z_t))^{-1}F(z_t)}} &\geq \frac{1}{12\norme{(\gamma^{-1}\bI + \nabla F(z_t))^{-1}}\cdot\norme{F(z_t)}}\nonumber\\ 
&\geq \frac{1}{12\norme{\nabla F(z_t)^{-1}}\cdot\norme{F(z_t)}}\nonumber\\
&= \frac{\sigma_{\min}(\nabla F(z_t))}{12\norme{F(z_t)}}\nonumber\\
&= \delta.\label{eq:q-lower-bound}
\end{align}
Thus, it follows that
 \begin{equation}
\frac{1}{16\norme{\hat{z}_t(\bar{\gamma}) - z_t}} \leq \bar{\gamma} \leq \frac{1}{8\norme{\hat{z}_t(\bar{\gamma}) - z_t}}
 \end{equation}
 for $\bar{\gamma} = \gamma_+$, as determined by Algorithm \ref{alg:gbs}.

\end{proof}

\begin{lemma}\label{lem:lipschitz-q}Let $q : \reals \mapsto \reals$ be defined as
\begin{equation}\label{eq:q}
q(\gamma) \defeq \frac{1}{12\gamma\norme{(\bI + \gamma \nabla F(z_t))^{-1}F(z_t)}} = \frac{1}{12\norme{(\gamma^{-1}\bI + \nabla F(z_t))^{-1}F(z_t)}},
\end{equation}
and let $\delta \defeq \frac{\sigma_{\min}(\nabla F(z_t))}{12\norme{F(z_t)}}$. Then, for all $\gamma \geq \delta$, we have
\begin{equation}
\abs{\frac{d}{d\gamma} q(\gamma)} \leq C,\quad \text{ for }\quad C \defeq \frac{1}{\delta^2}\pa{\frac{\frac{1}{\delta} + \norme{\nabla F(z_t)}}{12\sigma_{\min}(\nabla F(z_t))\norme{F(z_t)}}}^3.
\end{equation}
\end{lemma}
\begin{proof}
We begin by rewriting $q(\gamma)$ as
\begin{align*}
q(\gamma) &= \frac{1}{12}\pa{F(z_t)^\top(\gamma^{-1}\bI + \nabla F(z_t))^{{-1}^\top} (\gamma^{-1}\bI + \nabla F(z_t))^{-1} F(z_t)}^{-1/2}\\
&= \frac{1}{12}\pa{F(z_t)^\top(\gamma^{-1}\bI + \nabla F(z_t)^\top)^{-1} (\gamma^{-1}\bI + \nabla F(z_t))^{-1} F(z_t)}^{-1/2}
\end{align*}
Now, let $M_1(\gamma) \defeq (\gamma^{-1}\bI + \nabla F(z_t)^\top)^{-1}$ and $M_2(\gamma) \defeq (\gamma^{-1}\bI + \nabla F(z_t))^{-1}$. By standard matrix calculus, we may observe that
\begin{equation}
\frac{d}{d\gamma} M_1(\gamma) = \frac{1}{\gamma^2}(\gamma^{-1}\bI + \nabla F(z_t)^\top)^{-2} \ \text{ and }\ \frac{d}{d\gamma} M_2(\gamma) = \frac{1}{\gamma^2}(\gamma^{-1}\bI + \nabla F(z_t))^{-2}.
\end{equation}

It follows that
\begin{align*}
\frac{d}{d\gamma} q(\gamma) &= -\frac{1}{2}q(\gamma)^3 \cdot \Bigg(F(z_t)^\top(\gamma^{-1}\bI + \nabla F(z_t)^\top)^{-1}\pa{\frac{d}{d\gamma} M_2(\gamma)}F(z_t)\\
&\qquad \qquad + F(z_t)^\top\pa{\frac{d}{d\gamma} M_1(\gamma)}(\gamma^{-1}\bI + \nabla F(z_t))^{-1} F(z_t)\Bigg)\\
&= -\frac{1}{2\gamma^2}q(\gamma)^3 \cdot \Bigg(F(z_t)^\top(\gamma^{-1}\bI + \nabla F(z_t)^\top)^{-1}(\gamma^{-1}\bI + \nabla F(z_t))^{-2}F(z_t)\\
&\qquad \qquad + F(z_t)^\top(\gamma^{-1}\bI + \nabla F(z_t)^\top)^{-2}(\gamma^{-1}\bI + \nabla F(z_t))^{-1} F(z_t)\Bigg).
\end{align*}

Now, by standard norm inequalities, we have
\begin{align*}
\abs{\frac{d}{d\gamma} q(\gamma)} &\leq \frac{1}{2\gamma^2}\abs{q(\gamma)}^3 \Big(\norme{(\gamma^{-1}\bI + \nabla F(z_t)^\top)^{-1}}\cdot\norme{(\gamma^{-1}\bI + \nabla F(z_t))^{-1}}^{2} \\&\qquad\qquad\qquad\quad+ \norme{(\gamma^{-1}\bI + \nabla F(z_t)^\top)^{-1}}^{2}\cdot\norme{(\gamma^{-1}\bI + \nabla F(z_t))^{-1}}\Big)\norme{F(z_t)}^2\\
& = \frac{1}{\gamma^2}\abs{q(\gamma)}^3 \cdot \norme{(\gamma^{-1}\bI + \nabla F(z_t))^{-1}}^{3}\cdot\norme{F(z_t)}^2.
\end{align*}
Note that for all $\gamma \geq \delta$,
\begin{equation}
\abs{q(\gamma)} = \frac{1}{12\norme{(\gamma^{-1}\bI + \nabla F(z_t))^{-1} F(z_t)}} \leq \frac{\gamma^{-1} + \norme{\nabla F(z_t)}}{\norme{F(z_t)}} \leq \frac{\frac{1}{\delta} + \norme{\nabla F(z_t)}}{12\norme{F(z_t)}}
\end{equation}
and 
\begin{equation}
\norme{(\gamma^{-1}\bI + \nabla F(z_t))^{-1}} = \frac{1}{\sigma_{\min}(\gamma^{-1}\bI + \nabla F(z_t))} \leq \frac{1}{\sigma_{\min}(\nabla F(z_t))},
\end{equation} where the final inequality follows by Assumption~\ref{assum:invert}. Taken together, this gives us that
\begin{equation}
\abs{\frac{d}{d\gamma} q(\gamma)} \leq \frac{1}{\delta^2}\pa{\frac{\frac{1}{\delta} + \norme{\nabla F(z_t)}}{12\sigma_{\min}(\nabla F(z_t))\norme{F(z_t)}}}^3,
\end{equation}
and so the lemma follows.
\end{proof}

\begin{algorithm}[h]
	\caption{\BSg}
	
	\begin{algorithmic}\label{alg:gbs}
		\STATE {\bfseries Input:} $z_t$, $0 < \eps < 1$, $\gamma_-^{\text{init}}$, $\gamma_+^{\text{init}}$.
		\STATE Initialize $\gamma_- \gets \gamma_-^{\text{init}}$, $\gamma_+ \gets \gamma_+^{\text{init}}$, $\bar{\gamma} \gets \frac{\gamma_- + \gamma_+}{2}$
		\STATE Set $\delta = \frac{\sigma_{\min}(\nabla F(z_t))}{12\norme{F(z_t)}}$, $C = \frac{1}{\delta^2}\pa{\frac{\frac{1}{\delta} + \norme{\nabla F(z_t)}}{12\sigma_{\min}(\nabla F(z_t))\norme{F(z_t)}}}^3$, $\bar{C} = \max\braces{C, 1}$, $N = O(\log(\frac{\bar{C}T}{\delta}))$.
		\STATE Define $\hat{z}_{t}(\gamma) \defeq z_t - \gamma(\bI + \gamma\nabla F(z_t))^{-1}F(z_t)$
		\FOR{$k=0$ {\bfseries to} $N-1$}
		\STATE $D = \frac{1}{12\norme{\hat{z}_t(\bar{\gamma}) - z_t}}$
		\IF{$\bar{\gamma} \leq D$}
		\STATE $\gamma_- \gets \bar{\gamma}$
		\ELSE
		\STATE $\gamma_+ \gets \bar{\gamma}$
		\ENDIF
		\STATE $\bar{\gamma} = \frac{\gamma_- + \gamma_+}{2}$
		\ENDFOR
		\RETURN $\bar{\gamma} \gets \gamma_+$
	\end{algorithmic}
\end{algorithm}

\subsection{Invertibility concerns}\label{sec:inversion}
While the general setting of Algorithm \ref{alg:fomp} assumes $(\bI + \nabla F(z_t))$ is invertible, it turns out that for convex-concave games, this assumption is not necessary. In particular, the Jacobian of the vector field \eqref{eq:saddlevf} is
\begin{equation}
\nabla F(x,y) =   \begin{bmatrix}
\nabla_{xx}^2 \phi(x,y) & \nabla_{xy}^2 \phi(x,y)   \\
-\nabla_{yx}^2 \phi(x,y) & -\nabla_{yy}^2 \phi(x,y)
\end{bmatrix}.
\end{equation}
Note that there is a natural decomposition of $\nabla F(x,y)$ as a sum of a symmetric and an anti-symmetric matrix, namely
\begin{equation}\label{eq:phi-decomp}
\nabla F(x,y) =   \begin{bmatrix}
\nabla_{xx}^2 \phi(x,y) & \bzero   \\
\bzero & -\nabla_{yy}^2 \phi(x,y)
\end{bmatrix} + \begin{bmatrix}
\bzero & \nabla_{xy}^2 \phi(x,y)   \\
-\nabla_{yx}^2 \phi(x,y) & \bzero
\end{bmatrix}.
\end{equation}
The following is a useful lemma about the real part of eigenvalues of matrices, based on such a symmetric-asymmetric decomposition.
\begin{lemma}\label{lem:sym-asym}
	Let $M$ be a real matrix such that $M = S + A$, where $S$ is a symmetric real matrix and $A$ is an antisymmetric real matrix. If $S$ is nonsingular, then $M$ is nonsingular. Likewise, if $S$ is positive definite (or PSD), then the real part of eigenvalues of $M$ are positive (or non-negative).
\end{lemma}
\begin{proof}[Proof of Lemma~\ref{lem:sym-asym}]
	Let $v$ be an eigenvector of $M$ with eigenvalue $\lambda$ (these may both be complex). Let $v = v_r + iv_i$ and $\lambda = \lambda_r + i \lambda_i$ be the decompositions of $v$ and $\lambda$ into real and imaginary parts.
	\begin{align*}
	\lambda \norme{v} = v^* M v &= v^* S v + v^* A v\\
	&= (v_r - iv_i)^\top S (v_r + iv_i) + (v_r - iv_i)^\top A (v_r + iv_i)\\
	&= v_r^\top S v_r + v_i^\top S v_i + i(v_r^\top S v_i - v_i^\top S v_r) + v_r^\top A v_r + v_i^\top A v_i + i(v_r^\top A v_i - v_i^\top A v_r)
	\end{align*}
	Since $x^\top A x = 0$ for any antisymmetric matrix $A$, we have that $\lambda_r = \frac{1}{\norme{v}}(v_r^\top S v_r + v_i^\top S v_i)$, which implies the conclusions of the lemma. To see the fact about antisymmetric matrices, observe:
	\begin{align*}
	x^\top A x = x^\top A^\top x = -x^\top A x \iff 2x^\top A x = 0.
	\end{align*}
\end{proof}
By convexity and concavity of $\phi(x,y)$ in $x$ and $y$, respectively, we know that the symmetric part of \eqref{eq:phi-decomp} is PSD for all $z \in \Z$. It follows that, for all $t$, $(\bI + \nabla F(z_t))$ is positive definite, and therefore invertible. It may additionally be seen in this setting that $\sigma_{\min}(\gamma^{-1}\bI + \nabla F(z_t)) \geq \sigma_{\min}(\nabla F(z_t))$.

\bibliography{main}
\bibliographystyle{plainnat}

\appendix

\section{Proof of Lemma~\ref{lem:F_bound}}\label{app:F_bound}
\begin{lemma}\label{lem:F_bound}
	Assume $F$ is first-order $L_1$ smooth and $D(u,v) = \frac{1}{2}\norme{u-v}^2$.
\begin{align}
\norme{F(z_t)} \le 4\sqrt{t} L_1 \norme{z_1 - z^*}.
\end{align}
\end{lemma}

To prove Lemma~\ref{lem:F_bound}, we need the following lemma, which we prove in \Cref{app:ind_bounded_iters}.
\begin{lemma}\label{lem:ind_bounded_iters}
	Suppose $F:\R^n \ra \R^n$ is $p^{th}$-order $L_p$-smooth.
	Let $\{z_t\}_{t=1}^T$ be the iterates generated by \Cref{alg:main} and let $z^*$ be a solution to MVI($F,\Z$). Then for any $v\in \Z$,
	\begin{align}
		\norme{z_t - v}^2 \le 2t\paren{8D(z_1,z^*) + \norme{z_1 - v}^2}.
	\end{align}	
\end{lemma}

\begin{proof}[Proof of Lemma~\ref{lem:F_bound}]
By Assumption~\ref{as:mvi-soln}, we know there exists a $z^*$ such that \cref{eq:VI} holds. By Lemma~\ref{lem:VI_equal}, any such $z^*$ is also a solution to \cref{eq:strong_VI}, namely:
\begin{align}
	\forall z\in R^n, \inner{F(z^*), z^* - z} \le 0.
\end{align}
Since we are in the unconstrained setting, this implies that $F(z^*)=0$. Then we have:
\begin{align}
\norme{F(z_t)} = \norme{F(z_t) - F(z^*)} \le L_1 \norme{z_t - z^*}\label{eq:almost_f_bound}
\end{align}
where the inequality follows by the $L_1$ smoothness of $F$. By Lemma~\ref{lem:ind_bounded_iters}, we have
\begin{align}
\norme{z_t - z^*}\le \sqrt{2t\paren{4\norme{z_1 - z^*}^2 + \norme{z_1 - z^*}^2}} \le 4\sqrt{t}\norme{z_1 - z^*}.
\end{align}
Combining this with \cref{eq:almost_f_bound} gives the result.
\end{proof}

\subsection{Proof of Lemma~\ref{lem:ind_bounded_iters}}\label{app:ind_bounded_iters}
We will need the following two lemmas to prove Lemma~\ref{lem:ind_bounded_iters}:
	\begin{lemma}\label{lem:simple_a}
	Let $a_i > 0$ for $i \in [n]$. Then,
	\begin{align}
		\paren{\sum_{i=1}^n a_i}^2 \le n\sum_{i=1}^n a_i^2.
	\end{align}
\end{lemma}

\begin{lemma}\label{lem:bounded_iters}
	Let $z^*$ be the solution to MVI$(F,\Z)$. Then for the iterates $z_t$ of Algorithm \ref{alg:main} initialized at $z_1$, we have:
	\begin{align}
		\frac{1}{8}\sum\limits_{t=1}^T \norme{z_{t+1} - z_t}^2 \leq D(z^*,z_1).
	\end{align}
\end{lemma}
We prove in Lemma~\ref{lem:simple_a} in \Cref{app:simple_a}, and we prove Lemma~\ref{lem:bounded_iters} in \Cref{app:bounded_iters}.

\begin{proof}[Proof of Lemma~\ref{lem:ind_bounded_iters}]
	By the triangle inequality, we have:
	\begin{align}
		\norme{z_t - v}^2 &\le \paren{\sum_{\tau=1}^t \norme{z_\tau - z_{\tau+1}} + \norme{z_1 - v}}^2\\
		&\le (t+1)\paren{\sum_{\tau=1}^t  \norme{z_\tau - z_{\tau+1}}^2 + \norme{z_1 - v}^2}\label{eq:two_sums}
	\end{align}
	where the second inequality follows from using Lemma~\ref{lem:simple_a} with $a_i = \norme{z_i - z_{i+1}}$ for $i\in [t]$ and $a_{t+1} = \norme{z_1 - v}$. We then apply Lemma~\ref{lem:bounded_iters} to \cref{eq:two_sums} to get the result, using the fact that $t+1 \le 2t$.
\end{proof}

\subsection{Proof of Lemma~\ref{lem:simple_a}}\label{app:simple_a}
 \begin{proof}[Proof of Lemma~\ref{lem:simple_a}]
 	Let $a$ be the vector of $a_i$'s. We define the following power means:
 	\begin{align}
 		M_1(a) &= \frac{\sum_{i=1}^n a_i}{n}\\
 		M_2(a) &= \paren{\frac{\sum_{i=1}^n a_i^2}{n}}^{1/2}.
 	\end{align}
 	By the power mean inequality, we have $M_1(a) \le M_2(a)$, which gives the result.
 \end{proof}

\subsection{Proof of Lemma~\ref{lem:bounded_iters}}\label{app:bounded_iters}
\begin{proof}[Proof of Lemma~\ref{lem:bounded_iters}]
	This follows from two simple observations. First, note that:
	\begin{align}\label{eq:distance_squared}
		\sum\limits_{t=1}^T\norme{z_{t+1} - z_t}^2 \le \sum\limits_{t=1}^T (2\norme{z_{t+1} - \hat{z}_t}^2 + 2\norme{\hat{z}_t-z_t}^2).
	\end{align}
	Now, by Assumption~\ref{as:mvi-soln}, we know that each term of $\sum\limits_{t=1}^T \inner{\gamma_t F(\hat{z}_t), \hat{z}_t - z}$ is non-negative for some $z^* \in \Z$, namely the solution to MVI$(F,\Z)$. Combining this with
	Lemma~\ref{lem:higher-order_mp} and \cref{eq:distance_squared} gives the result.
\end{proof}

\section{Equivalence of exact solutions to weak and strong MVIs}\label{app:weak_VI}
\begin{lemma}[\cite{kinderlehrer1980introduction}] \label{lem:VI_equal}
For continuous $F:\R^n \ra \R^n$, any solution of \cref{eq:VI} is a solution to \cref{eq:strong_VI}.
\end{lemma}
\begin{proof}
	Let $z^*$ be a solution to \cref{eq:VI}. Let $z = z^* + t(v-z^*)$ for an arbitrary $v\in \Z$ and for $t> 0$. Then:
	\begin{align}
		\inner{F(z^* + t(v-z^*)), -t(v-z^*)} &\le 0\\
		\iff \inner{F(z^* + t(v-z^*)), z^* - v)} &\le 0 \label{eq:to_limit}.
	\end{align}
	Taking the limit of \cref{eq:to_limit} as $t$ goes to $0$ gives:
	\begin{align}
		\inner{F(z^*),z^* - v} \le 0.
	\end{align}
Thus, $z^*$ is a solution to \cref{eq:strong_VI}.
\end{proof}

\end{document}

%% file: defs.tex
\def\norm#1{\mathopen\| #1 \mathclose\|}

\newcommand{\ignore}[1]{}

\def\reals{{\mathbb R}}

\def\bzero{\mathbf{0}}

\def\bold0{\mathbf{0}}

\def\bI{\mathbf{I}}

\def\bI{\mathbf{I}}

\def\eps{\varepsilon}
\def\epsilon{\varepsilon}

\newcommand{\defeq}{\stackrel{\text{def}}{=}}

\newcommand{\braces}[1]{\left\{#1\right\}}
\newcommand{\pa}[1]{\left(#1\right)}

\newcommand{\bra}[1]{\left[#1\right]}
\newcommand{\abs}[1]{\left|#1\right|}

\DeclareMathOperator*{\argmin}{arg\,min}

\newcommand\R{\mathbb{R}}

\newcommand\Y{\mathcal{Y}}

\newcommand\X{\mathcal{X}}
\newcommand\Z{\mathcal{Z}}

\newcommand\y{\mathbf{y}}

\newcommand\inner[1]{\langle #1 \rangle}

\newcommand\grad{\nabla}

\newcommand{\comment}[1]{}
\newcommand{\ra}{\rightarrow}
\crefformat{equation}{#2(#1)#3}
\newtheorem{assumption}[theorem]{Assumption}
\newtheorem*{theorem*}{Theorem}
\newtheorem{fact}[theorem]{Fact}
\newcommand{\paren}[1]{\left(#1\right)}
\newcommand{\norme}[1]{\norm{#1}_2}